\documentclass[12pt,centertags,oneside]{amsart}

\usepackage{amsmath,amstext,amsthm,amscd,typearea,hyperref,stmaryrd,mathabx}
\usepackage{amssymb}
\usepackage{a4wide}
\usepackage[mathscr]{eucal}
\usepackage{mathrsfs}
\usepackage{typearea}
\usepackage{charter}
\usepackage{pdfsync}
\usepackage[a4paper,width=16.2cm,top=3cm,bottom=3cm,lines=50]{geometry}
\usepackage[all]{xy}
\usepackage{tikz}
\usepackage{enumitem}
\usepackage[foot]{amsaddr}
\usepackage{academicons}
\usepackage{xcolor}

\numberwithin{equation}{section}
\newcommand{\orcid}[1]{\href{https://orcid.org/#1}{\textsc{orc}i\textsc{d}}}

\title[]{Hyperbolic entropy for harmonic measures on singular holomorphic foliations}

\author{Fran\c cois Bacher}
\address{Universit\'e de Lille, 
Laboratoire de math\'ematiques Paul Painlev\'e, 
CNRS U.M.R. 8524,  
59655 Villeneuve d'Ascq Cedex, 
France.}
\email{francois.bacher@univ-lille.fr}

\date{\today}

\def\restriction#1#2{\mathchoice
              {\setbox1\hbox{${\displaystyle #1}_{\scriptstyle #2}$}
              \restrictionaux{#1}{#2}}
              {\setbox1\hbox{${\textstyle #1}_{\scriptstyle #2}$}
              \restrictionaux{#1}{#2}}
              {\setbox1\hbox{${\scriptstyle #1}_{\scriptscriptstyle #2}$}
              \restrictionaux{#1}{#2}}
              {\setbox1\hbox{${\scriptscriptstyle #1}_{\scriptscriptstyle #2}$}
              \restrictionaux{#1}{#2}}}
\def\restrictionaux#1#2{{#1\,\smash{\vrule height .8\ht1 depth .85\dp1}}_{\,#2}}

\theoremstyle{plain}
\newtheorem{thm}{Theorem}[section]
\newtheorem{lem}[thm]{Lemma}

\newtheorem{prop}[thm]{Proposition}
\newtheorem{cor}[thm]{Corollary}
\newtheorem*{thm*}{Theorem}
\newtheorem*{conj*}{Conjecture}

\theoremstyle{definition}
\newtheorem{defn}[thm]{Definition}

\newtheorem*{exmp*}{Example}

\theoremstyle{remark}

\makeatletter

\@addtoreset{equation}{section}
\makeatother

\DeclareMathOperator{\diam}{diam}

\DeclareMathOperator{\Leb}{Leb}

\DeclareMathOperator{\sing}{sing}

\DeclareMathOperator{\Vol}{Vol}

\DeclareMathOperator{\Aut}{Aut}

\newcommand{\cjg}[1]{\overline{#1}}

\newcommand{\adh}[1]{\overline{#1}}

\newcommand{\PC}{P}

\newcommand{\eps}{\varepsilon}

\newcommand{\fol}{\mathscr{F}}

\newcommand{\leafatlas}{\mathscr{L}}

\newcommand{\nsplfol}{\left(\mani{M},\leafatlas\right)}

\newcommand{\plfol}{\left(\mani{M},\leafatlas,\mani{E}\right)}

\newcommand{\ddc}{dd^{\text{c}}}

\newcommand{\set}[1]{\mathbb{#1}}

\newcommand{\leaf}{L}

\newcommand{\leafu}[1]{\leaf_{#1}}

\newcommand{\norm}[1]{\left\Vert#1\right\Vert}

\newcommand{\intcc}[2]{\left[#1,#2\right]}

\newcommand{\intoo}[2]{\left(#1,#2\right)}

\newcommand{\intco}[2]{\left[#1,#2\right)}

\newcommand{\intent}[2]{\left\llbracket#1,#2\right\rrbracket}

\newcommand{\Cmod}[1]{\left\vert#1\right\vert}

\newcommand{\proj}[1]{\set{P}^{#1}}

\newcommand{\mani}[1]{#1}

\newcommand{\manis}[2]{\mani{#1}\setminus\mani{#2}}

\newcommand{\dhimpsing}[2]{\dhimpnov(#2,\mani{#1})}

\newcommand{\rD}[1]{#1\set{D}}

\newcommand{\DR}[1]{\set{D}_{#1}}

\newcommand{\adhDR}[1]{\adh{\set{D}}_{#1}}

\newcommand{\wo}[2]{#1\backslash#2}

\newcommand{\dhimp}[2]{\dhimpnov(#1,#2)}

\newcommand{\dhermfnov}[1]{d_{\leafu{#1}}}

\newcommand{\dhermf}[3]{\dhermfnov{#1}(#2,#3)}

\newcommand{\dhimpnov}{d}

\newcommand{\dhimpsnov}[1]{\dhimpnov_{#1}}

\newcommand{\dhimps}[3]{\dhimpsnov{#1}(#2,#3)}

\newcommand{\dPC}[2]{\dPCnov(#1,#2)}

\newcommand{\dPCp}[2]{\dPCnov\left(#1,#2\right)}

\newcommand{\dPCs}[3]{d_{\PC,#1}(#2,#3)}

\newcommand{\dPCnov}{d_{\PC}}

\newcommand{\metPC}{g_{\PC}}

\newcommand{\metm}[1]{g_{\mani{#1}}}

\newcommand{\textfol}{holomorphic foliation}

\newcommand{\textsingfol}{singular \textfol{}}

\newcommand{\textsingfols}{\textsingfol{}s}

\newcommand{\wt}[1]{\widetilde{#1}}

\newcommand{\foldPnC}[2]{\fol_{#1}(\proj{#2})}

\begin{document}

\theoremstyle{plain}

\begin{abstract}  Let $\fol=\plfol$ be a Brody-hyperbolic \textsingfol{} on a compact complex manifold~$\mani{M}$. Suppose that~$\fol$ has isolated singularities and that its Poincar\'e metric is complete. This is the case for a very large class of singularities, namely, non-degenerate and saddle-nodes in dimension~$2$. Let~$\mu$ be an ergodic harmonic measure on~$\fol$. We show that the upper and lower local hyperbolic entropies of~$\mu$ are leafwise constant almost everywhere. Moreover, we show that the entropy of~$\mu$ is at least~$2$.
  \end{abstract}

\maketitle

\section{Introduction}

The dynamical theory for laminations by Riemann surfaces has recently received much attention. In particular, a lot of progress have been focused on the case of hyperbolic leaves. For singular holomorphic foliations on~$\proj{n}$, this is a typical setup. Indeed, every polynomial vector field on~$\set{C}^n$ induces a \textsingfol{} that can be compactified onto~$\proj{n}$. This foliation is always singular. For $d,n\in\set{N}$ with $n\geq2$, denote by $\foldPnC{d}{n}$ the space of \textsingfols{} of degree~$d$ on~$\proj{n}$. As is shown by Lins~Neto and Soares~\cite{LNS}, a generic foliation $\fol\in\foldPnC{d}{n}$ has only non-degenerate singularities. This result is based on a previous analysis of the Jouanolou foliations~\cite{Jou}. By the results of Lins~Neto~\cite{LN,LN2} and Glutsyuk~\cite{Glu}, such a foliation is hyperbolic if $d\geq2$ and even Brody-hyperbolic in the sense of~\cite{DNSII}. In the case $n=2$, Nguy\^{e}n also uses the integrability of the holonomy cocycle~\cite{Nguholo} to compute the Lyapunov exponent~\cite{NguLyap} of a generic foliation in $\foldPnC{d}{2}$. We refer the reader to the survey articles~\cite{surDinhSib,surForSib,surVANG18,surVANG21} for more details about these questions.

Solving heat equations with respect to a harmonic current, Dinh, Nguy\^{e}n and Sibony prove in~\cite{DNS12} ergodic theorems for laminations. In particular, they obtain an effective and geometric analog of Birkhoff's Theorem. Somehow, the comparison relies on considering the time to be the hyperbolic distance in a uniformization of a leaf by the Poincar\'e disk~$\set{D}$. Therefore, in a series of two articles~\cite{DNSI,DNSII}, the three authors study the modulus of continuity of the Poincar\'e metric of compact laminations and of \textsingfols{} with linearizable singularities. They also introduce various notions of entropy for hyperbolic laminations. These concepts, together with Lyapunov exponents~\cite{NguOse}, could bring a fruitful bridge between the dynamics of foliations and the dynamics of maps. Our works~\cite{Bac1,Bac2,Bac3} are devoted to generalize some of their results about the Poincar\'e metric, the heat diffusions and the topological entropy to foliations with non-degenerate singularities. Here, we are interested in the measure-theoretic hyperbolic entropy.

Let us introduce some notations to state our results. Let $\fol=\plfol$ be a \textsingfol{}. Here, $\mani{M}$ is a compact complex manifold, $\leafatlas$ is a foliated atlas of~$\manis{M}{E}$ and~$\mani{E}$ is the singular set of~$\fol$. Suppose that the leaves of~$\fol$ are hyperbolic Riemann surfaces and for $x\in\manis{M}{E}$, fix $\phi_x\colon\set{D}\to\leafu{x}$ a uniformization of the leaf~$\leafu{x}$. Such a map is unique up to precomposing by a rotation. Let~$\metm{M}$ be a Hermitian metric on~$\mani{M}$ and denote by~$\dhimpnov$ the induced distance on~$\mani{M}$. Since we consider the hyperbolic distance in~$\set{D}$ to be a time parametrizing~$\leafu{x}$, 
\[\label{defBowenintro}d_R(x,y)=\inf_{\theta\in\set{R}}\sup_{\xi\in\adhDR{R}}\dhimp{\phi_x(\xi)}{\phi_y(e^{i\theta}\xi)},\quad x,y\in\manis{M}{E},\]
for $R>0$, can be thought as a \emph{Bowen distance} up to time~$R$~\cite[pp.~581--582]{DNSI}. Here,~$\DR{R}$ denotes the disk of center~$0$ and hyperbolic radius~$R$. The infimum over~$\theta$ means that we consider distances between the two closest parametrizations of the leaves~$\leafu{x}$ and~$\leafu{y}$. That is, we consider the distance up to reparametrization. Therefore, we obtain a notion of Bowen balls and of \emph{topological entropy} $h(\fol)$, which is a number that measures heuristically the exponential growth with~$R$ of the amount of Bowen balls needed to cover~$\manis{M}{E}$. The more precise definition of $h(\fol)$ will be given in Section~\ref{secdefent}. There are some cases for which it is known the topological entropy is finite. Namely, for smooth compact laminations without singularities~\cite{DNSI} and foliations with only non-degenerate singularities on compact complex surfaces~\cite{DNSII,Bac3}. Here, we show a lower bound for the entropy.

If moreover~$\fol$ is endowed with a harmonic measure~$\mu$, one can ask more weakly how many balls are needed to cover a subset of measure at least $1-\delta$, for $\delta\to0$, and we obtain a \emph{measure-theoretic entropy} $h(\mu)$. It is clear that $h(\mu)\leq h(\fol)$. To study these numbers, Dinh, Nguy\^{e}n and Sibony~\cite{DNSI} introduce other notions of entropy, and in particular \emph{local upper and lower entropies} $h^{\pm}(\mu,x)$, that measure the exponential decay with $R$ of the $\mu$-measure of the Bowen balls centered in~$x$. In the case without singularities, they can show the following.

\begin{thm}[Dinh--Nguy\^{e}n--Sibony~{\cite[Theorem~4.2]{DNSI}}]\label{nshpmconstant} Let $\fol=\nsplfol{}$ be a smooth compact lamination by hyperbolic Riemann surfaces and~$\mu$ be a harmonic measure on~$\fol$. Then, the local upper and lower entropies $x\mapsto h^{\pm}(\mu,x)$ are leafwise constant. In particular, if~$\mu$ is ergodic, they are constant $\mu$-almost everywhere and we denote them by $h^{\pm}(\mu)$. Moreover, $h(\mu)\geq2$.
\end{thm}

Here, we improve their results for \textsingfols{}. 

\begin{thm}\label{mainthm}  Let $\fol=\plfol$ be a Brody-hyperbolic singular holomorphic foliation on a compact complex manifold~$\mani{M}$. Let~$\mu$ be an ergodic harmonic measure on~$\fol$. Suppose that the singularities of~$\fol$ are isolated and that the Poincar\'e metric of~$\fol$ is complete. Then, $x\mapsto h^{\pm}(\mu,x)$ are constant $\mu$-almost everywhere. Moreover, $h(\mu)\geq2$.
\end{thm}

As we show in Section~\ref{secdefent}, if~$\fol$ is Brody-hyperbolic, the assumption on the Poincar\'e metric is only local near the singularities. Therefore, by local works of \mbox{Canille}~\mbox{Martins} and Lins~Neto~\cite{MLN}, our result can be applied to a large class of singularities.

\begin{cor}\label{maincor}  Let $\fol=\plfol$ be a Brody-hyperbolic singular holomorphic foliation on a compact complex manifold~$\mani{M}$. Let~$\mu$ be an ergodic harmonic measure on~$\fol$. Suppose that all the singularities $p\in\mani{E}$ of~$\fol$ are of one of the following type.
  \begin{enumerate}
  \item In a chart centered at~$p$, $\fol$ is generated by a vector field $X=\sum_{j=k}^{+\infty}X_j$, with the $X_j$ homogeneous vector fields of degree $j$, for $j\geq k$, and $X_k$ admitting an isolated singularity. In particular, if $k=1$,~$p$ is a non-degenerate singularity.
  \item The point~$p$ is a saddle-node singularity of $\fol$ (in ambient dimension~$2$).
  \end{enumerate}
  Then, the conclusions of Theorem~\ref{mainthm} hold.
\end{cor}

Let us explain the method of our proof. First, we follow the idea of Dinh, Nguy\^{e}n and Sibony of considering transversal entropies $\wt{h}^{\pm}(\mu,x)$. The main advantage is that it is not so difficult to show that they are leafwise constant using the quasi-invariance of~$\mu$. As in the non-singular case, we reduce our problem to proving $h^{\pm}=\wt{h}^{\pm}+2$. Somehow, this means showing that the plaquewise entropy is always~$2$. That is, we have to determine the size of the intersection of a Bowen ball with a plaque. In the non-singular case, estimates on the distance between the identity and $\tau\in\Aut(\set{D})$ on~$\DR{R}$ show that this intersection is essentially a disk of radius $\simeq e^{-R}$. Indeed, if there is a~$\xi\in\DR{R}$ such that $\dPC{\xi}{\tau(\xi)}$ is big, then there are intermediate points where it is not so big to be in the same plaque, and still bigger than~$\eps$. In that case, the Poincar\'e distance and the distance in~$\mani{M}$ are comparable. This simple argument does not work anymore in our context because of the singularities. When the above distance is not so big, we could be near the singularity. In that case, we can not conclude anything. On the other hand, when we are outside neighbourhoods of the singularities, the Poincar\'e distance between two points of the same leaf could be big and the distance in~$\mani{M}$ very small if they do not belong to the same plaque.

Let us describe how we overcome this difficulty. First, we use \mbox{Dinh}, \mbox{Nguy\^{e}n} and \mbox{Sibony}'s \mbox{Birkhoff} type Theorem~\cite{DNS12} to control the time spent near the singularities. Next, we prove thorough estimates on the distance between the identity and an automorphism~$\tau$ of~$\set{D}$ and find a large amount of points where the distances between them is big enough but not too big. Then, this enables to show that the distances between two points of the intersection of a Bowen ball with a plaque correspond in some sense to large subsets of the disk which are spent near the singular set. Since this time is small, this should not happen often. More precisely, we show somehow that this intersection contains at most~$e^{\delta R}$ disks of radius $e^{-R}$, for~$\delta$ comparable to the proportion of time spent near the singular set. Since this~$\delta$ can be chosen arbitrarily small, we have our result. The last argument relies on studying what we call \emph{subsets of $\set{R}^n$ with prescribed steps}. That is, subsets for which it is known that the distance between any two points is in some small union of intervals. These look like what we obtain when constructing the Cantor ternary set by induction.

The article is organized as follows. In Section~\ref{secdefent}, we begin by recalling some basic facts about the Poincar\'e metric. We show that Corollary~\ref{maincor} is indeed a consequence of Theorem~\ref{mainthm}. Then, we define the different entropies we shall use. We recall some results of the three authors to reduce to studying the plaquewise entropy and showing Proposition~\ref{hpm=hpmT+2}. In Section~\ref{secproof}, we prove this proposition. We begin by reducing to two lemmas that will allow to work in one single plaque. Next, we study briefly subsets of $\set{R}^n$ with prescribed steps. We finish by introducing a fine geometric setup and proving the two lemmas.

\subsection*{Notation} Throughout this paper, we denote by $\set{D}$ the unit disk of $\set{C}$, and $\rD{r}$ the open disk of radius $r\in\set{R}_+^*$ for the standard Euclidean metric of $\set{C}$. For $R\in\set{R}_+^*$, we also denote by $\DR{R}$ the open disk of hyperbolic radius $R$ in $\set{D}$, so that $\DR{R}=\rD{r}$ with $r=\frac{e^R-1}{e^R+1}$, or if $r\in\intco{0}{1}$, with $R=\ln\frac{1+r}{1-r}$.

If $\fol{}=\plfol{}$ is a \textsingfol{} and $x\in\manis{M}{E}$, we denote by~$\leafu{x}$ the leaf of $\fol{}$ through $x$. Moreover, if $\leafu{x}$ is hyperbolic, we denote by $\phi_{x}\colon\set{D}\to\leafu{x}$ a uniformization of $\leafu{x}$ such that $\phi_x(0)=x$. Given a Hermitian metric $\metm{M}$ on~$\mani{M}$, we denote by $\dhimpnov$ the distance induced by $\metm{M}$ and by $\set{B}(a,\rho)$ the ball of center~$a$ and radius~$\rho$. We use the same notation for a ball in $\set{R}^n$. Similarly, we denote by $\dPCnov$ the Poincar\'{e} distance if one is given a Poincar\'{e} metric denoted $\metPC{}$ on the disk or on a leaf. We are also given Bowen distances~$d_R$ on $\manis{M}{E}$, and denote by $B_R(x,\eps)$ the ball of center~$x$ and radius~$\eps$ for the distance~$d_R$. 

Finally, we denote by $\Leb(B)$ the Lebesgue measure of a Borel set $B\subset\set{R}^n$. We use $C$, $C'$, $C''$, etc. to denote positive constants which may change from a line to another.

\subsection*{Acknowledgments} The  author is supported by the Labex CEMPI (ANR-11-LABX-0007-01) and by the project QuaSiDy (ANR-21-CE40-0016).

\section{Hyperbolic entropy for foliations}\label{secdefent}

\subsection{Leafwise Poincar\'e metric}

In all this section, we let $\fol=\plfol$ be a \textsingfol{} on a compact complex manifold $\mani{M}$. Suppose that $\mani{M}$ is endowed with a Hermitian metric $\metm{M}$ and for $x\in\manis{M}{E}$, consider
\begin{equation}\label{eqdefeta}\eta(x)=\sup\left\{\norm{\alpha'(0)}_{\metm{M}};\,\alpha\colon\set{D}\to\leafu{x}~\text{holomorphic}~\text{such}~\text{that}~\alpha(0)=x\right\}.\end{equation}
Above, $\norm{v}_{\metm{M}}$ is the norm of a vector $v\in T_x\leafu{x}$ with respect to the Hermitian metric $\metm{M}$. That is, $\norm{v}_{\metm{M}}=\left(g_{\mani{M},x}(v,v)\right)^{1/2}$. The map~$\eta$ was introduced by Verjovsky in~\cite{Ver}. It is designed to satisfy the following facts.

\begin{prop}\label{propeta}
  \begin{enumerate}
  \item For $x\in\manis{M}{E}$, $\eta(x)<+\infty$ if and only if the leaf $\leafu{x}$ is hyperbolic, that is, it is uniformized by the Poincar\'e disk $\set{D}$.

   \item If $\leafu{x}$ is hyperbolic, we have $\eta(x)=\norm{\phi'(0)}_{\metm{M}}$, where $\phi\colon\set{D}\to\leafu{x}$ is any uniformization of $\leafu{x}$ such that $\phi(0)=x$.
  \item If $\leafu{x}$ is hyperbolic, then $\frac{4\metm{M}}{\eta^2}$ induces the Poincar\'e metric on $\leafu{x}$.
  \end{enumerate}
\end{prop}

Here, we are interested in the case of hyperbolic leaves. We need to specify our global setting. Let us introduce some terminology, for which we follow~\cite[Definition~3.1]{DNSII}.

\begin{defn} \label{defBrody}If all the leaves of $\fol$ are hyperbolic, we say that $\fol$ is \emph{hyperbolic}. If moreover there exists a constant $c_0>0$ such that $\eta(x)<c_0$ for all $x\in\manis{M}{E}$, we say that~$\fol$ is \emph{Brody-hyperbolic}.

  We say \emph{the Poincar\'e metric of $\fol$ is complete} if $\fol$ is hyperbolic and $\frac{4\metm{M}}{\eta^2}$ is complete as a Hermitian metric on $\manis{M}{E}$. We refer the reader to~\cite[Property~{\bf P.2.}]{MLN} for more details.
\end{defn}

In full generality, it is not known whether the map~$\eta$ is continuous. In fact, we can only say that it is lower semi-continuous~\cite[Theorem~20]{surForSib}. However, with the asumption that the Poincar\'e metric is complete, Lins Neto~\cite[Theorem~A]{LN} shows the continuity of~$\eta$. For discussions about this regularity, see~\cite[Section~4]{surVANG21} and the references therein. See also~\cite{DNSI,DNSII,Bac1} for a modulus of continuity with more restrictive singularities.

As mentioned before, there are cases when we know the Poincar\'e metric is complete.

\begin{thm}[Canille~Martins--Lins~Neto~{\cite[Theorems~1 and~2]{MLN}}]\label{thmCMLN} Let~$\fol$ be a local \textsingfol{} near $0\in\set{C}^n$, with singularity at~$0$ of any type listed in Corollary~\ref{maincor}. Then, the Poincar\'e metric of~$\fol$ is complete at~$0$.
\end{thm}

Therefore, Corollary~\ref{maincor} is a consequence of the following, the proof of which is essentially the same as~\cite[Proposition~3.3]{DNSII}.

\begin{lem}\label{metPCcompletelocimpglob} Let $\fol=\plfol$ be a Brody-hyperbolic \textsingfol{} on a compact complex manifold. Suppose that all $a\in\mani{E}$ admit a neighbourhood~$U_a$ on which the Poincar\'e metric of $\restriction{\fol}{U_a}$ is complete at $a$. Then, the Poincar\'e metric of~$\fol$ is complete.
\end{lem}

\begin{proof} Let $\eta_a$ denote the~$\eta$ map of the foliation $\restriction{\fol}{U_a}$. We will show that there exists a constant $C_a>0$ such that $\eta\leq C_a\eta_a$ on some neighbourhood $V_a\subset U_a$ of~$a$. Since~$\mani{M}$ is compact and~$\fol$ is Brody-hyperbolic, it is not difficult to see that it implies the result. Let $\rho>0$ be such that the ball $\set{B}(a,\rho)$ is contained in~$U_a$ and consider $V_a=\set{B}\left(a,\frac{\rho}{2}\right)$. Fix $x\in V_a$ and $\phi_x\colon\set{D}\to\leafu{x}$ a uniformization of $\leafu{x}$ such that $\phi_x(0)=x$. That way, $\eta(x)=\norm{\phi_x'(0)}_{\metm{M}}$. Since $\fol$ is Brody-hyperbolic, there exists a radius $r_0\in\intoo{0}{1}$, independent on~$x$, such that $\phi_x(\rD{r_0})\subset U_a$. More precisely, the hyperbolic radius $R_0=\ln\frac{1+r_0}{1-r_0}$ corresponding to $r_0$ can be chosen to be equal to $\frac{\rho}{c_0}$, for $c_0>0$ as in Definition~\ref{defBrody}. By the extremal condition~\eqref{eqdefeta} of~$\eta_a$, it follows that $\eta_a(x)\geq r_0\norm{\phi_x'(0)}_{\metm{M}}=r_0\eta(x)$.
\end{proof}

\subsection{Bowen distance and various notions of entropy}

From now on, suppose that $\fol$ is hyperbolic. Also, denote by $\dhimpnov{}$ the distance on $\mani{M}$ induced by the ambient Hermitian metric $\metm{M}$. For $x\in\manis{M}{E}$, fix $\phi_x\colon\set{D}\to\leafu{x}$ a uniformization of $\leafu{x}$ such that $\phi_x(0)=x$. The idea of Dinh, Nguy\^{e}n and Sibony~\cite[pp.~581--582]{DNSI} is to consider the Poincar\'e distance in $\set{D}$ to be a canonical time. More precisely, for $R\geq0$, consider the Bowen distance
\[d_R(x,y)=\inf\limits_{\theta\in\set{R}}\sup\limits_{\xi\in\adhDR{R}}\dhimp{\phi_x(\xi)}{\phi_y(e^{i\theta}\xi)},\quad x,y\in\manis{M}{E}.\]
It measures the distance between the orbits of $x$ and $y$ up to time $R$. It is clear that it is independent on the choice of $\phi_x$. This enables us to define various notions of entropy. For $x\in\manis{M}{E}$, $R\geq0$ and $\eps>0$, denote by $B_R(x,\eps)=\{y\in\mani{M}\,;\,d_R(x,y)<\eps\}$ the \emph{Bowen ball} of radius~$\eps$ and center~$x$ up to time $R$. For $Y\subset\manis{M}{E}$, $R\geq0$, $\eps>0$ and $F\subset Y$, we  say that $F$ is \emph{$(R,\eps)$-dense} in $Y$ if $Y\subset\cup_{x\in F}B_R(x,\eps)$. Denote by $N(Y,R,\eps)$ the minimal cardinality of an $(R,\eps)$-dense subset in $Y$. The \emph{topological hyperbolic entropy} of $Y$ is defined as
\[h(Y)=\sup\limits_{\eps>0}\limsup\limits_{R\to+\infty}\frac{1}{R}\log N(Y,R,\eps).\]
For $Y=\manis{M}{E}$, we denote it by $h(\fol)$. If $\mani{M}$ is compact, it is not difficult to prove that it does not depend on the choice of $\metm{M}$. A similar and equivalent definition can be made with maximal $(R,\eps)$-separated sets, but we do not need it. The interested reader can see~\cite[Section~3]{DNSI} for more details.

Now, we introduce some entropies for harmonic measures following~\cite[Section~4]{DNSI}. Basically, harmonic measures are quasi-invariant measures by the foliation~$\fol$. Here, we only use~\eqref{desintmu} below, which can be seen as a local expression of this quasi-invariance. That is, a harmonic measure can be locally disintegrated into a transverse positive Radon measure times a leafwise positive harmonic density. We refer the reader to~\cite{DNS12,surForSib,Gar,surVANG21} for a definition and a link with $\ddc$-closed currents.

To define measure-theoretic entropy, the idea is to consider analogous quantities as those emerging from Brin--Katok Theorem. Let~$\mu$ be a harmonic probability measure on~$\fol$. All our harmonic measures will be supposed to have total mass~$1$. For $R\geq0$, $\eps>0$ and $\delta\in\intoo{0}{1}$, let $N(R,\eps,\delta)$ be the smallest integer $N$ such that there exists $x_1,\dots,x_N\in\manis{M}{E}$ with $\mu\left(\cup_{i=1}^NB_R\left(x_i,\eps\right)\right)\geq1-\delta$. Define the \emph{hyperbolic entropy of~$\mu$} as
\[h(\mu)=\sup\limits_{\delta\in\intoo{0}{1}}\sup\limits_{\eps>0}\limsup\limits_{R\to+\infty}\frac{1}{R}\log N(R,\eps,\delta).\]
To study this quantity, Dinh, Nguy\^{e}n and Sibony consider \emph{local upper and lower hyperbolic entropies} $h^{\pm}(\mu,x)$, for $x\in\manis{M}{E}$. Namely,
\[\begin{aligned}h^+(\mu,x,\eps)=\limsup\limits_{R\to+\infty}-\frac{1}{R}\log\mu\left(B_R(x,\eps)\right),&\qquad h^+(\mu,x)=\sup_{\eps>0}h^+(\mu,x,\eps);\\
h^-(\mu,x,\eps)=\liminf\limits_{R\to+\infty}-\frac{1}{R}\log\mu\left(B_R(x,\eps)\right),&\qquad h^-(\mu,x)=\sup_{\eps>0}h^-(\mu,x,\eps).\end{aligned}\]
In the context of a non-singular foliation, they prove that these are leafwise constant (see Theorem~\ref{nshpmconstant}). This enables them to obtain a link between all these entropies.

\begin{thm}[Dinh--Nguy\^{e}n--Sibony~{\cite[Proposition~4.5]{DNSI}}]\label{orderentropies} Let $\fol=\nsplfol{}$ be a (non-singular) hyperbolic holomorphic foliation on a complex manifold~$\mani{M}$ (we do not assume that~$\mani{M}$ is compact). Let $\mu$ be a harmonic measure on~$\fol$. If the quantities $x\mapsto h^{\pm}(\mu,x)$ are constant $\mu$-almost everywhere, then,
  \[h^-(\mu)\leq h(\mu)\leq h^+(\mu)\leq h(\fol).\]
\end{thm}

Under these hypotheses, the above theorem is in fact implicitly proved, for the authors state it in the setup of Theorem~\ref{nshpmconstant}. However, they only use that $h^{\pm}(\mu,x)$ are constant. Let us define another notion of entropy that they introduce to prove their results.

\subsection{Local transversal entropy} Let $x\in\manis{M}{E}$. If $\eps>0$ is sufficiently small, then there exists a flow box $U\simeq\set{D}\times\set{T}$ such that $\set{B}(x,\eps)\subset U$. Denote by $\pi_{\set{T}}\colon U\to\set{T}$ the projection on the second coordinate. Consider a disintegration of the measure $\mu$ in~$U$. That is,
\begin{equation}\label{desintmu}\mu=\int_{\set{T}}\left(\int_{\set{D}}f_t(z)\metPC(z,t)\right)d\nu(t),\end{equation}
where $\nu$ is a finite positive Radon measure on $\set{T}$ and the $f_t$, for $t\in\set{T}$, are positive harmonic functions on~$\set{D}$ with $f_t(0)=1$. Such a decomposition is unique up to changing the $f_t$ on a $\nu$-negligible set and exists due to~\cite[Propositions~2.3 and~5.1]{DNS12}. Define
\[\begin{aligned}\wt{h}^+(\mu,x,\eps)=\limsup\limits_{R\to+\infty}-\frac{1}{R}\log\nu\left(\pi_{\set{T}}\left(B_R(x,\eps)\right)\right),&\qquad \wt{h}^+(\mu,x)=\sup_{\eps>0}\wt{h}^+(\mu,x,\eps);\\
\wt{h}^-(\mu,x,\eps)=\liminf\limits_{R\to+\infty}-\frac{1}{R}\log\nu\left(\pi_{\set{T}}\left(B_R(x,\eps)\right)\right),&\qquad \wt{h}^-(\mu,x)=\sup_{\eps>0}\wt{h}^-(\mu,x,\eps).\end{aligned}\]
It is not difficult to show that $\wt{h}^{\pm}(\mu,x)$ do not depend on the choice of the flow box $U$. The following is implicitly proved by the three authors. They state it with more restrictive hypotheses but their proof still works without changing anything.
\begin{thm}[Dinh--Nguy\^{e}n--Sibony~{\cite[Theorem~4.2]{DNSI}}]\label{hpmTconstant} Let $\fol=\nsplfol$ be a Brody-hyperbolic holomorphic foliation on a complex manifold~$\mani{M}$ and~$\mu$ be a harmonic measure on~$\fol$. Then, the quantities $x\mapsto \wt{h}^{\pm}(\mu,x)$ are leafwise constant. In particular, if~$\mu$ is ergodic, they are constant $\mu$-almost everywhere and we denote them by $\wt{h}^{\pm}(\mu)$.
\end{thm}

Note that there is no compactness assumption on~$\mani{M}$. Therefore, we can apply it to $\left(\manis{M}{E},\leafatlas\right)$ in the singular case. Now, Theorem~\ref{mainthm} wil be a consequence of Theorems~\ref{orderentropies} and~\ref{hpmTconstant}, combined with the following, the proof of which occupies the next section and needs some detours.

\begin{prop}\label{hpm=hpmT+2} Let $\fol=\plfol$ be a Brody-hyperbolic singular holomorphic foliation on a compact complex manifold~$\mani{M}$. Let also~$\mu$ be an ergodic harmonic measure on~$\fol$. Suppose that the singularities of~$\fol$ are isolated and that the Poincar\'e metric of~$\fol$ is complete in the sense of Definition~\ref{defBrody}. Then, for $\mu$-almost every $x\in\manis{M}{E}$,
  \[h^\pm(\mu,x)=\wt{h}^{\pm}(\mu,x)+2.\]
  In particular, $x\mapsto h^{\pm}(\mu,x)$ are constant $\mu$-almost everywhere and $h^-(\mu)\geq2$.
\end{prop}

\section{Proof of Proposition~\ref{hpm=hpmT+2}}\label{secproof}

\subsection{First reduction} Let $\fol=\plfol$ be a Brody-hyperbolic singular holomorphic foliation on a compact complex manifold~$\mani{M}$ and~$\mu$ be an ergodic harmonic measure on~$\fol$. Suppose that the singularities of~$\fol$ are isolated and that the Poincar\'e metric of~$\fol$ is complete. First, we intend to reduce the proof to two lemmas, which together will give estimates to the measure of plaquewise Bowen balls. More precisely, we estimate the quantity of automorphisms~$\tau$ of the disk, such that $\phi\circ\tau$ is close to~$\phi$, for some given~$\phi$ uniformization of a leaf. These estimates will take some time to be proven, but we already show how to obtain Proposition~\ref{hpm=hpmT+2} from them.

  \begin{lem}\label{estimLebtheta} For $\delta>0$ and $\mu$-almost every $x\in\manis{M}{E}$, there exists $\eps_0>0$ satisfying the following. For $\eps>0$ sufficiently small, there exists $C>0$ such that for all sufficiently large $R>0$ and $y\in B_R(x,\eps_0)$,
    \[C^{-1}e^{-R}\leq\Leb\left(\left\{\theta\in\intcc{-\pi}{\pi};~\dhimps{\adhDR{R}}{\phi_y\circ r_{\theta}}{\phi_y}<\eps\right\}\right)\leq Ce^{-(1-\delta)R}.\]
    Here, $r_{\theta}\in\Aut(\set{D})$ denotes the rotation of angle~$\theta$.
    \end{lem}

    We need a similar result for a general automorphism in $\Aut(\set{D})$. For $\zeta\in\set{D}$, denote by $\tau_{\zeta}\in\Aut(\set{D})$ defined by $\tau_{\zeta}(\xi)=\frac{\xi+\zeta}{1+\cjg{\zeta}\xi}$. We need some preparation.

\begin{lem}\label{flowboxUV} Given $x\in\manis{M}{E}$, there exist flow boxes $\set{D}\times\set{T}\simeq U\subset V$ containing~$x$ and $r_0\in\intoo{0}{1}$ such that
  \[U\subset\bigcup\limits_{t\in\set{T}}\phi_t\left(\rD{\frac{r_0}{3}}\right)\quad\text{and}\quad\bigcup\limits_{y\in U}\phi_y(\rD{r_0})\subset V.\]
\end{lem}

\begin{proof} First fix a flow box~$V$ containing~$x$. Since $\eta$ is bounded above, we can find~$r_0\in\intoo{0}{1}$ and~$U$ sufficiently small such that the map
  \[\rD{r_0}\times U\to\mani{M},\qquad (\zeta,y)\mapsto\phi_y(\zeta)\]
  is injective and with values in~$V$. Now, since~$\eta$ is bounded from below near~$x$, shrinking~$U$ if necessary, we can suppose $U\subset\cup_{t\in\set{T}}\phi_t\left(\rD{\frac{r_0}{3}}\right)$. 
\end{proof}

\begin{lem}\label{estimLebzeta} For $\delta>0$ and $\mu$-almost every $x\in\manis{M}{E}$, there exists $\eps_0>0$ satisfying the following. For $\eps>0$ sufficiently small, there exists $C>0$ such that for all sufficiently large $R>0$ and $y\in B_R(x,\eps_0)$,
    \[C^{-1}e^{-3R}\leq\Leb\left(\left\{(\zeta,\theta)\in\rD{r_0}\times\intcc{-\pi}{\pi};~\dhimps{\adhDR{R}}{\phi_y\circ\tau_{\zeta}\circ r_{\theta}}{\phi_y}<\eps\right\}\right)\leq Ce^{-3(1-\delta)R}.\]
      Here, $r_0$ is given by the previous lemma.
\end{lem}

\begin{proof}[Proof of Proposition~\ref{hpm=hpmT+2}.] We take for granted Lemmas~\ref{estimLebtheta} and~\ref{estimLebzeta} and show how they give us our result. Consider~$x\in\manis{M}{E}$ satisfying the conclusions of the lemmas and fix $\delta>0$. Let $r_0\in\intoo{0}{1}$, $\set{D}\times\set{T}\simeq U\subset V$ be given by Lemma~\ref{flowboxUV} and $0<\eps<\frac{\eps_0}{3}$ sufficiently small to have the conclusions of both other lemmas. Using~\eqref{desintmu}, we get
  \[\mu(B_R(x,\eps))=\int_{\set{T}}\left(\int_{\set{D}}f_t(z)\chi_{B_R(x,\eps)}(z,t)\metPC(z,t)\right)d\nu(t),\]
  where $\chi_B$ denotes the characteristic function of a Borel set~$B$.  Now, since $x\in\manis{M}{E}$ is still far from the singular set, $\metPC$ is equivalent to the Lebesgue measure and by the Harnack inequality, there exists $c>0$ such that $c^{-1}\leq f_t\leq c$ on $U$. Therefore, there exists $c'>1$ with
  \[c'^{-1}\mu(B_R(x,\eps))\leq\int_{\pi_{\set{T}}(B_R(x,\eps))}\Leb\left(B_R(x,\eps)\cap(\set{D}\times\{t\})\right)d\nu(t)\leq c'\mu(B_R(x,\eps)).\]
  Here, the Lebesgue measure is the one of~$\set{D}$ in the flow box $U$. For $t\in\pi_{\set{T}}(B_R(x,\eps))$, choose $y_t\in B_R(x,\eps)\cap(\set{D}\times\{t\})$. Then, $B_R(x,\eps)\cap(\set{D}\times\{t\})\subset B_R\left(y_t,2\eps\right)\cap(\set{D}\times\{t\})$. Moreover, for $y_t\in B_R\left(x,\frac{\eps}{2}\right)$, $B_R\left(y_t,\frac{\eps}{2}\right)\cap(\set{D}\times\{t\})\subset B_R(x,\eps)\cap(\set{D}\times\{t\})$. Hence,
  \begin{equation}\label{compmunu}\begin{aligned}\mu(B_R(x,\eps))&\leq c'\nu(\pi_{\set{T}}(B_R(x,\eps)))\sup_{t\in\pi_{\set{T}}(B_R(x,\eps))}\Leb(B_R(y_t,2\eps)\cap(\set{D}\times\{t\}));\\
  \mu(B_R(x,\eps))&\geq c'^{-1}\nu\left(\pi_{\set{T}}\left(B_R\left(x,\eps/2\right)\right)\right)\inf_{t\in\pi_{\set{T}}\left(B_R\left(x,\eps/2\right)\right)}\Leb\left(B_R\left(y_t,\eps/2\right)\cap\left(\set{D}\times\{t\}\right)\right).\end{aligned}\end{equation}
  So, changing~$\eps$ if necessary, it will be sufficient to bound above and below the Lebesgue measure of $B_R(y,\eps)\cap(\set{D}\times\{t\})$, for $y\in B_R(x,\eps_0)\cap(\set{D}\times\{t\})$. Let $y'\in B_R(y,\eps)\cap(\set{D}\times\{t\})$. Since~$y$ and~$y'$ are in the same plaque, there is $\zeta\in\rD{r_0}$ such that $y'=\phi_y(\zeta)$. Thus,
  \[B_R(y,\eps)\cap(\set{D}\times\{t\})=\left\{\phi_y(\zeta);\zeta\in\rD{r_0},\exists\theta\in\intcc{-\pi}{\pi},\dhimps{\adhDR{R}}{\phi_y\circ\tau_{\zeta}\circ r_{\theta}}{\phi_y}<\eps\right\}.\]
  Denote by
  \[\begin{aligned}A_{\eps}(y)&=\left\{\zeta\in\rD{r_0},\exists\theta\in\intcc{-\pi}{\pi},\dhimps{\adhDR{R}}{\phi_y\circ\tau_{\zeta}\circ r_{\theta}}{\phi_y}<\eps\right\},\\
      B_{\eps}(y)&=\left\{(\zeta,\theta)\in\rD{r_0}\times\intcc{-\pi}{\pi},\dhimps{\adhDR{R}}{\phi_y\circ\tau_{\zeta}\circ r_{\theta}}{\phi_y}<\eps\right\},\\
      C_{\eps}(y')&=\left\{\theta\in\intcc{-\pi}{\pi},\dhimps{\adhDR{R}}{\phi_{y'}\circ r_{\theta}}{\phi_{y'}}<\eps\right\},\quad y'\in B_R(y,\eps)\cap\left(\set{D}\times\{t\}\right).\end{aligned}\]
Since the Poincar\'e metric in the flow box is equivalent to the Lebesgue measure, we have
\begin{equation}\label{compAepsBR}c^{-1}\Leb\left(A_{\eps}(y)\right)\leq\Leb\left(B_R(y,\eps)\cap(\set{D}\times\{t\})\right)\leq c\Leb\left(A_{\eps}(y)\right).\end{equation}
Moreover,
\[\begin{aligned}\Leb\left(B_{\eps}(y)\right)&\leq\Leb(A_{\eps}(y))\sup_{\xi\in\pi_{\zeta}(B_{\eps}(y))}\Leb(C_{2\eps}(\phi_y(\xi))),\\
\Leb\left(B_{2\eps}(y)\right)&\geq\Leb\left(A_{\eps}(y)\right)\inf_{\xi\in\pi_{\zeta}\left(B_{\eps}(y)\right)}\Leb\left(C_{\eps}(\phi_y(\xi))\right),\end{aligned}\]
where $\pi_{\zeta}$ is the projection on the disk from $\set{D}\times\intcc{-\pi}{\pi}$. Now, applying Lemmas~\ref{estimLebtheta} and~\ref{estimLebzeta} and coming back to~\eqref{compAepsBR}, we get
\[C^{-1}e^{-(2+\delta)R}\leq\Leb\left(B_R(y,\eps)\cap(\set{D}\times\{t\})\right)\leq Ce^{-(2-3\delta)R}.\]
By~\eqref{compmunu}, we obtain
\[(c'C)^{-1}\mu(B_R(x,\eps))e^{(2-3\delta)R}\leq\nu(\pi_{\set{T}}(B_R(x,\eps)))\leq c'Ce^{(2+\delta)R}\mu(B_R(x,\eps)).\]
Letting $R$ go to infinity and noting that $\eps$ can be chosen arbitrarily small, this implies
\[h^{\pm}(\mu,x)-2-\delta\leq\wt{h}^{\pm}(\mu,x)\leq h^{\pm}(\mu,x)-2+3\delta.\]
Since~$\delta$ was chosen arbitrarily, we conclude the proof.
\end{proof}

\subsection{Subsets of $\set{R}^n$ with prescribed steps} This subsection is devoted to studying briefly subsets of $\set{R}^n$ with distances in a pre-defined set. As will appear in the next subsection, the sets appearing in Lemmas~\ref{estimLebtheta} and~\ref{estimLebzeta} will be of that kind (for $n=1$ or~$3$) and we are especially interested in estimating their Lebesgue measure. The sets we study look like the union of intervals we obtain when constructing Cantor sets by induction. Let us be more precise.

\begin{defn} Let $N\in\set{N}$ and for $i\in\intent{0}{N}$, $j\in\{1,2\}$, take $\eps_{i,j}\geq0$. Suppose that
  \[0<\eps_{i+1,2}<\eps_{i,1}<\eps_{i,2},\qquad i\in\intent{0}{N-1},\]
  and $\eps_{N,1}=0$. A Borel subset $A\subset\set{R}^n$ is said to have \emph{prescribed $(\eps_{i,j})_{i,j}$-steps} if for any $x,y\in A$, there exists $i\in\intent{0}{N}$ such that $\norm{x-y}\in\intcc{\eps_{i,1}}{\eps_{i,2}}$.
\end{defn}

In particular, it should be noted that $\diam(A)\leq\eps_{0,2}$. The following gives an upper bound to the Lebesgue measure of those sets. The bound is loose but will be sufficient.

\begin{lem}\label{mesepspresc} Let $A$ be a Borel subset of $\set{R}^n$ with prescribed $(\eps_{i,j})_{i\in\intent{0}{N},j\in\{1,2\}}$-steps. Suppose that for $i\in\intent{1}{N}$, $\eps_{i,2}<\frac{\eps_{i-1,1}}{2}$. Then, if $V_n=\Vol(\set{B}(0,1))$ in $\set{R}^n$,
  \[\Leb(A)\leq V_n\eps_{N,2}^n3^{nN}\times\prod\limits_{i=0}^{N-1}\left(\frac{\eps_{i,2}}{\eps_{i,1}}\right)^n.\]
\end{lem}

\begin{proof} Let us show by induction on~$N$ that such an~$A$ is contained in the union of at most $p_N=3^{nN}\times\prod_{i=0}^{N-1}\left(\frac{\eps_{i,2}}{\eps_{i,1}}\right)^n$ balls of radius $\eps_{N,2}$. If $N=0$, this is trivial. Suppose that this was proven at rank $N-1$. Take~$A$ as in the statement of the lemma. Then, we can apply the induction hypothesis to $\eps_{i,j}$, $i\in\intent{0}{N-2},j\in\{1,2\}$, $\eps_{N-1,2}$ and $\eps'_{N-1,1}=0$. This gives $p_{N-1}$ balls of radius $\eps_{N-1,2}$ which cover~$A$. Now, each of these $p_{N-1}$ balls has prescribed $(\eps_{i,j})_{i\in\{N-1,N\},j\in\{1,2\}}$-steps. Indeed, if $B$ is one of these balls and if $x,y\in A\cap B$, then
  \[\norm{x-y}\in\intcc{0}{2\eps_{N-1,2}}\cap\left(\bigcup\limits_{i=0}^N\intcc{\eps_{i,1}}{\eps_{i,2}}\right)=\intcc{0}{\eps_{N,2}}\cup\intcc{\eps_{N-1,1}}{\eps_{N-1,2}}.\]
    Here, we use that $\eps_{N-1,2}<\frac{\eps_{N-2,1}}{2}$. Since $p_N=p_{N-1}p'_1$, with $p'_1=3^n\left(\frac{\eps_{N-1,2}}{\eps_{N-1,1}}\right)^n$, this is enough to prove the induction hypothesis for $N=1$. This means that for $x,y\in A$, either $\norm{x-y}\leq\eps_{1,2}$, or $\norm{x-y}\in\intcc{\eps_{0,1}}{\eps_{0,2}}$. Let $x_1,\dots,x_r$ be a maximal $\eps_{1,2}$-separated subset of~$A$. Since the $x_i$ are $\eps_{1,2}$-separated, $\norm{x_i-x_j}>\eps_{1,2}$ for $i\neq j$ and the prescribed steps imply that $\norm{x_i-x_j}\geq\eps_{0,1}$. Therefore, the balls $\left(\set{B}\left(x_i,\frac{\eps_{0,1}}{2}\right)\right)_{i\in\intent{1}{r}}$ are disjoint. Since $\diam(A)\leq\eps_{0,2}$, they are also all included in $\set{B}\left(x_1,\eps_{0,2}+\frac{\eps_{0,1}}{2}\right)$. Hence,
    \[\begin{aligned}V_n\left(\eps_{0,2}+\frac{\eps_{0,1}}{2}\right)^n=\Vol\left(\set{B}\left(x_1,\eps_{0,2}+\frac{\eps_{0,1}}{2}\right)\right)&\geq\Vol\left(\bigcup\limits_{i=1}^r\set{B}\left(x_i,\frac{\eps_{0,1}}{2}\right)\right)\\
        &\geq\sum\limits_{i=1}^r\Vol\left(\set{B}\left(x_i,\frac{\eps_{0,1}}{2}\right)\right)=rV_n\left(\frac{\eps_{0,1}}{2}\right)^n,\end{aligned}\]
  using first the inclusion and second the disjointness of the balls. From the above inequality, it follows that $r\leq\left(\frac{2\eps_{0,2}+\eps_{0,1}}{\eps_{0,1}}\right)^n\leq\left(3\frac{\eps_{0,2}}{\eps_{0,1}}\right)^n=p_1$. Moreover, by maximality of the $\eps_{1,2}$-separated family $x_1,\dots,x_r$, $A\subset\cup_{i=1}^r\set{B}(x_i,\eps_{1,2})$. This completes the proof.
\end{proof}

\subsection{Geometric setup} In this subsection, we prepare some geometric and dynamical ground. In particular, we need Dinh, Nguy\^{e}n and Sibony's Birkhoff-type Theorem. Let $\fol=\plfol$ be a Brody-hyperbolic \textsingfol{} on a compact complex manifold. Suppose that~$\fol$ has isolated singularities. For $x\in\manis{M}{E}$, fix $\phi_x\colon\set{D}\to\leafu{x}$ a uniformization of $\leafu{x}$. For $r\in\intoo{0}{1}$ and the corresponding $R=\ln\frac{1+r}{1-r}$, define
\[m_{x,R}=\frac{1}{M_R}(\phi_x)_*\left(\log^+\frac{r}{\Cmod{\zeta}}\metPC\right),\]
where $\log^+$ stands for $\max(\log,0)$, $\metPC$ denotes the Poincar\'e metric on the disk and
\[M_R=\int_{\set{D}}\log^+\frac{r}{\Cmod{\zeta}}\metPC=-2\pi\log(1-r^2)\sim_{R\to+\infty} 2\pi R.\]

\begin{thm}[Dinh--Nguy\^{e}n--Sibony~{\cite[Theorem~7.1]{DNS12}} (see also~{\cite[Theorem~5.36]{surVANG21}})] Keep the above notations and hypotheses. Let~$\mu$ be an ergodic harmonic measure on~$\fol$. Then, for $\mu$-almost every $x\in\manis{M}{E}$, $m_{x,R}$ converges weakly to~$\mu$ when~$R$ tends to infinity.
\end{thm}

Now, fix an $x\in\manis{M}{E}$ such that $m_{x,R}\to\mu$. We use this convergence to control the time that $\leafu{x}$ spends near the singular set. More precisely, we need the following.

\begin{lem}\label{lemsetupsing} Suppose moreover that the Poincar\'e metric of~$\fol$ is complete. Fix $\delta>0$. There exist $0<\rho_1<\rho_2<\rho_3$ satisfying the following. Denote by $U_j^{\sing}=\{y\in\manis{M}{E};~\dhimpsing{E}{y}\leq\rho_j\}$, for $j\in\{1,2,3\}$.
  \begin{enumerate}[label=(\roman*),ref=\roman*]
  \item \label{rho1rho2sep} If $y\in\manis{M}{E}$ and $j\in\{1,2\}$ are such that $\dhimpsing{E}{y}\leq\rho_j$ and $y'\in\leafu{y}$ is such that $\dhimpsing{E}{y'}\geq\rho_{j+1}$, then $\dPC{y}{y'}\geq4$.
  \item \label{rho2petit} $\mu\left(U_3^{\sing}\right)<\frac{\delta}{2}$. In particular, for all large enough $R>0$, $m_{x,R}\left(U_3^{\sing}\right)<\delta$.
    \item \label{x>3} $x\notin U_3^{\sing}$.
  \end{enumerate}
\end{lem}

\begin{proof} Choose first $\rho_3>0$ satisfying~\eqref{rho2petit} and~\eqref{x>3}. Clearly we can do so since $\mu(\mani{E})=0$ and $m_{x,R}\to\mu$. Then, since the Poincar\'e metric is complete, we can choose $\rho_2\in\intoo{0}{\rho_3}$ and after that $\rho_1\in\intoo{0}{\rho_2}$ satisfying~\eqref{rho1rho2sep}.
\end{proof}

\begin{lem}\label{lemdefeps0} With the notations of Lemma~\ref{lemsetupsing}, there exist $c>1$ and $\eps_0>0$ such that the following holds. Let $y\notin\frac{1}{2}U_1^{\sing}$, $y'\in\leafu{y}$ be such that $\dPC{y}{y'}\leq\eps_0$. Then,
  \[c^{-1}\dPC{y}{y'}\leq\dhimp{y}{y'}\leq c\dPC{y}{y'}.\]
  Above, $\frac{1}{2}U_1^{\sing}=\left\{y\in\manis{M}{E};~\dhimpsing{E}{y}\leq\frac{\rho_1}{2}\right\}$.
\end{lem}

\begin{proof} Given $\rho_1$, we can cover $\wo{\mani{M}}{\frac{1}{4}U_1^{\sing}}$ by a finite number of flow boxes. It follows that~$\eta$ is bounded by $c^{-1}<\eta<c$ on $\wo{\mani{M}}{\frac{1}{2}U_1^{\sing}}$. Moreover, we find an $\eps_0>0$ such that $y\notin\frac{1}{4}U^{\sing}_1$ implies that $\set{B}(y,c\eps_0)$ is contained in a flow box. For~$y$ and~$y'$ satisfying the above conditions, we get
  \[\frac{c\eps_0}{2}\geq\frac{c}{2}\dPC{y}{y'}\geq\dhermf{y}{y}{y'},\]
  where $\dhermfnov{y}$ denotes the distance on $\leafu{y}$ induced by the restriction of~$\metm{M}$ to~$\leafu{y}$. This implies that~$y$ and~$y'$ are in the same plaque and the bounds on~$\eta$ give us both needed inequalities.
\end{proof}

The following describes circles that are in a large part mapped to the singular set. These $R_{i,j}$ will prove to give prescribed steps of the sets involved in Lemmas~\ref{estimLebtheta} and~\ref{estimLebzeta}.

\begin{lem}\label{lemRj12}With the notations of Lemma~\ref{lemsetupsing}, take $R>0$ sufficiently large. Denote by
  \[I_j^{\sing}=\left\{R'\in\intcc{0}{R};~\Leb\left\{\xi\in\partial\DR{R'};\phi_x(\xi)\in U_j^{\sing}\right\}>\frac{\pi}{3}\right\},\quad j\in\{1,2,3\}.\]
  Here, $\Leb$ of the circles $\partial\DR{R'}$ are normalized to have mass $2\pi$. There exists  $\left(R_{i,j}\right)_{i\in\intent{0}{N},j\in\{1,2\}}$, such that $4<R_{i,2}<R_{i,1}<R_{i+1,2}$, $i\in\intent{0}{N-1}$, $R\geq R_{N,2}$ and $R_{N,1}=+\infty$, with
    \begin{enumerate}[label=(\arabic*),ref=\arabic*]
    \item \label{Rj12} $R_{i,1}-R_{i,2}\geq4$, $R_{i+1,2}-R_{i,1}\geq4$, for $i\in\intent{0}{N-1}$.
    \item \label{Rjsing1}$I_1^{\sing}\subset\cup_{i=0}^N\intoo{R_{i,2}}{R_{i,1}}$. 
      \item\label{Rjsing3}$\intcc{0}{R}\cap\left(\cup_{i=0}^N\intoo{R_{i,2}}{R_{i,1}}\right)\subset I_3^{\sing}$.
    \end{enumerate}
    Moreover, $\sum_{i=0}^{N-1}R_{i,1}-R_{i,2}+(R-R_{N,2})\leq12\delta R$ and $N\leq3\delta R$.
  \end{lem}

  Big $R$ will correspond to small $\eps$, that is why the order is reversed compared to sets with prescribed steps.

  \begin{proof} \renewcommand{\qedsymbol}{}This is where we use our weird conditions on $\rho_1,\rho_2,\rho_3$. The $I_j^{\sing}$ are open subsets of $\intcc{0}{R}$. Hence, they are an at most countable union of intervals. For each $C_3$ connected component of $I_3^{\sing}$, if $C_3\cap I_1^{\sing}\neq\emptyset$, consider $I_{C_3}$ to be the convex hull of the union of all connected components~$C_2$ of $I_2^{\sing}\cap C_3$. Define $I$ to be the union of the $I_{C_3}$, for all the $C_3$. Since each $I_{C_3}$ contains a point of $I_1^{\sing}$ and all points of $I_2^{\sing}$ in $C_3$, condition~\eqref{rho1rho2sep} implies that the length of $I_{C_3}$ is at least~$4$. Indeed, points in the part of the circle of hyperbolic radius~$R_1\in I_1^{\sing}$ which are in $U_1^{\sing}$ must go out of $U_2^{\sing}$ so that we leave $I_{C_3}$. Moreover, since we work component by component in $I_3^{\sing}$, the same condition~\eqref{rho1rho2sep} implies that the $I_{C_3}$ are pairwise distant of at least~$4$. Indeed, points in the part of the circle of hyperbolic radius~$R_2\in I_{C_3}$ which are in $U_2^{\sing}$ must go out of $U_3^{\sing}$ so that we enter a new $I_{C_3'}$. If we denote the $I_{C_3}$ by $\intoo{R_{i,2}}{R_{i,1}}$, points~\eqref{Rj12},~\eqref{Rjsing1} and~\eqref{Rjsing3} are then clear. To have $R_{0,2}>4$, we use~\eqref{x>3}. The bounds on $R_{N,j}$ are just technical for further notations. Here actually, it is only asked that $R_{N,2}-R_{N-1,1}\geq4$. If this is not satisfied for $R_{N,2}=R$, then we can just put $R_{N-1,1}=+\infty$ and still get all points of the lemma.
  \end{proof}
  
For the estimate $\sum_{i=0}^{N-1}R_{i,1}-R_{i,2}+R-R_{N,2}\leq12\delta R$, we need the following computation.

\begin{lem}\label{estimintegral} Let $0<r_2<r_1<r$ and $R_j=\ln\frac{1+r_j}{1-r_j}$ (or $r_j=\frac{e^{R_j}-1}{e^{R_j}+1}$), $j\in\{1,2\}$ be the corresponding hyperbolic radii. Then
  \[\int_{r_2}^{r_1}\log\left(\frac{r}{\rho}\right)\frac{4\rho}{\left(1-\rho^2\right)^2}d\rho=R_1-R_2+o(1),\]
  where the $o(1)$ stands for $R_2$ goes to infinity.
\end{lem}

\begin{proof}[End of proof of Lemma~\ref{lemRj12}] Taking for granted the above estimate, let us finish our proof. If~$R$ is sufficiently large, we have $m_{x,R}\left(U_3^{\sing}\right)<\delta$ by~\eqref{rho2petit}. We get
  \[M_R\delta>\int_{\set{D}}\chi_{\left\{\phi_x\in U_3^{\sing}\right\}}(\xi)\log^+\left(\frac{r}{\Cmod{\xi}}\right)\metPC(\xi),\]
  where as usual, $\chi_B$ is the characteristic function of~$B$. Let us continue our computation.
  \[\begin{aligned}M_R\delta&>\int_{0}^r\left(\int_0^{2\pi}\chi_{\left\{\phi_x\in U_3^{\sing}\right\}}(\rho e^{i\theta})d\theta\right)\log\left(\frac{r}{\rho}\right)\frac{4\rho d\rho}{\left(1-\rho^2\right)^2}\\
      M_R\delta&>\frac{\pi}{3}\sum_{i=0}^{N-1}\int_{r_{i,2}}^{r_{i,1}}\log\left(\frac{r}{\rho}\right)\frac{4\rho d\rho}{\left(1-\rho^2\right)^2}+\frac{\pi}{3}\int_{r_{N,2}}^r\log\left(\frac{r}{\rho}\right)\frac{4\rho d\rho}{\left(1-\rho^2\right)^2}\\
      M_R\delta&>\frac{\pi}{3}\sum_{i=0}^{N-1}\left(R_{i,1}-R_{i,2}\right)+\frac{\pi}{3}\left(R-R_{N,2}\right)+o(R).\end{aligned}\]
Here, we have denoted $r_{i,j}=\frac{e^{R_{i,j}}-1}{e^{R_{i,j}}+1}$ as usual and we have used~\eqref{Rjsing3}. For the last inequality, we applied Lemma~\ref{estimintegral} and Ces\`aro Theorem. Note that~\eqref{Rj12} implies $R_{i,2}>8i+4$, which goes to infinity with~$i\leq N=O(R)$. Now, since $M_R\sim2\pi R$, we get our result.\end{proof}

\begin{proof}[Proof of Lemma~\ref{estimintegral}] This is only a silly computation. Integrating by parts,
  \[\int_{r_2}^{r_1}\log\left(\frac{r}{\rho}\right)\frac{4\rho}{\left(1-\rho^2\right)^2}d\rho=2\left[\frac{1}{1-\rho^2}\log\left(\frac{r}{\rho}\right)\right]_{r_2}^{r_1}+2\int_{r_2}^{r_1}\frac{d\rho}{\rho(1-\rho^2)}.\]
  Since $\frac{1}{\rho(1-\rho^2)}=\frac{1}{2}\left(\frac{1}{1-\rho}-\frac{1}{1+\rho}+\frac{2}{\rho}\right)$, we obtain
  \[\int_{r_2}^{r_1}\log\left(\frac{r}{\rho}\right)\frac{4\rho}{\left(1-\rho^2\right)^2}d\rho=2\left[\frac{\rho^2}{1-\rho^2}\log\left(\frac{r}{\rho}\right)\right]_{r_2}^{r_1}-\log(1-r_1^2)+\log(1-r_2^2).\]
  Now, we translate into hyperbolic language. For $j\in\{1,2\}$, $\frac{r_j^2}{1-r_j^2}=\frac{1}{4}e^{R_j}+o\left(e^{R_j}\right)$. Moreover, $\log\frac{r}{r_j}=2e^{-R_j}-2e^{-R}+o\left(e^{-R_j}\right)$. So the first bracket above is $o(1)$. The other terms satisfy $1-r_j^2=e^{-R_j}(4+o(1))$. We get the estimate wanted.
\end{proof}
    
\subsection{End of proof of Lemmas~\ref{estimLebtheta} and~\ref{estimLebzeta}} Our preparation is soon to be over and we move on to the end of our proof. We begin by Lemma~\ref{estimLebtheta}, for it involves less computation and already explains the ideas of the proof of Lemma~\ref{estimLebzeta}. First, let us study the distance to identity of a rotation in $\Aut(\set{D})$.

\begin{lem}\label{distrotid} Let $0<\eps_1<\eps_2$ be sufficiently small with $\eps_1\leq\frac{1}{16}\eps_2$, $\theta\in\intcc{-\pi}{\pi}$ and $\xi\in\set{D}$. Then, if $\Cmod{\xi}\geq\frac{1}{4}$,
  \[\dPCp{\xi}{e^{i\theta}\xi}\in\intcc{\eps_1}{\eps_2},\quad\text{for}\quad 8\eps_2^{-1}\Cmod{\sin(\theta/2)}\leq1-\Cmod{\xi}^2\leq\frac{1}{2}\eps_1^{-1}\Cmod{\sin(\theta/2)}.\]
\end{lem}

\begin{proof} Denote by $\tanh^{-1}(x)=\frac{1}{2}\ln\frac{1+x}{1-x}$, $x\in\intoo{-1}{1}$ and compute.
  \[\begin{aligned}\dPCp{\xi}{e^{i\theta}\xi}&=2\tanh^{-1}\Cmod{\frac{\xi(1-e^{i\theta})}{1-e^{i\theta}\Cmod{\xi}^2}}=2\tanh^{-1}\Cmod{\frac{2\xi\sin(\theta/2)}{\cos(\theta/2)\left(1-\Cmod{\xi}^2\right)-i\sin(\theta/2)\left(1+\Cmod{\xi}^2\right)}}\\
      &=2\tanh^{-1}\left(\frac{2\Cmod{\xi}\Cmod{\sin(\theta/2)}}{\left(\left(1-\Cmod{\xi}^2\right)^2+4\Cmod{\xi}^2\sin^2(\theta/2)\right)^{1/2}}\right)=f\left(\frac{1-\Cmod{\xi}^2}{2\Cmod{\xi}\Cmod{\sin(\theta/2)}}\right),\end{aligned}\]
  for $f\colon\set{R}_+^*\to\set{R}_+^*$ defined by $f(x)=2\tanh^{-1}\left(\frac{1}{\sqrt{1+x^2}}\right)$. The function~$f$ is strictly decreasing and has inverse $f^{-1}\colon y\mapsto\sqrt{\frac{1}{\tanh^{2}(y/2)}-1}=\frac{2}{y}+O_{y\to0}(1)$, as is shown by a straightforward computation. It follows, that if $\eps_1,\eps_2$ are sufficiently small, and if $x\in\intcc{4\eps_2^{-1}}{\eps_1^{-1}}$, then $x\in\intcc{f^{-1}(\eps_2)}{f^{-1}(\eps_1)}$ and $f(x)\in\intcc{\eps_1}{\eps_2}$. The bounds on $1-\Cmod{\xi}^2$ exactly give
  \[\frac{1-\Cmod{\xi}^2}{2\Cmod{\xi}\Cmod{\sin(\theta/2)}}\in\intcc{4\eps_2^{-1}}{\eps_1^{-1}},\]
  for $\Cmod{\xi}\leq\frac{1}{4}$, so we get our result.
\end{proof}

\begin{lem}\label{ensthetapasprescrits} Take the notations of Lemmas~\ref{lemdefeps0} and~\ref{lemRj12} and let $\eps>0$ be sufficiently small, $R>0$ be sufficiently large. Define $\eps_{i,1}=2\arcsin\left(8c\eps e^{-R_{i,1}}\right)$ and $\eps_{i,2}=2\arcsin\left(16c\eps e^{-R_{i,2}}\right)$, for $i\in\intent{0}{N}$. Note that $\eps_{N,1}=0$. Then, for $y\in B_R(x,\eps_0)$,
  \[A=\left\{\theta\in\intcc{-\pi}{\pi};~\dhimps{\adhDR{R}}{\phi_y\circ r_{\theta}}{\phi_y}<\eps\right\}\]
  has prescribed $(\eps_{i,j})_{i,j}$-steps.
\end{lem}

\begin{proof}Let $y\in B_R(x,\eps_0)$, $\eps$ be sufficiently small to determine and $\theta_1,\theta_2\in A$. We have to show that $\Cmod{\theta_1-\theta_2}\in\cup_{i=0}^N\intcc{\eps_{i,1}}{\eps_{i,2}}$. Since $\theta_1,\theta_2\in A$, we have
  \begin{equation}\label{eqtheta12}\dhimps{\adhDR{R}}{\phi_y\circ r_{\theta_1}}{\phi_y\circ r_{\theta_2}}=\dhimps{\adhDR{R}}{\phi_y\circ r_{\Cmod{\theta_1-\theta_2}}}{\phi_y}<2\eps.\end{equation}
  Denote by $\theta=\Cmod{\theta_1-\theta_2}$. With the notations of Lemma~\ref{lemdefeps0}, set $\eps_2<\eps_0$, $\eps_1=\frac{1}{16}\eps_0$ sufficiently small, and $\eps=\frac{1}{2}c^{-1}\eps_1$. Fix $\phi_y$ a uniformization of $\leafu{y}$ with $\dhimps{\adhDR{R}}{\phi_x}{\phi_y}<\eps_0$. By definition, it is clear that $\eps_0<\frac{\rho_1}{2}$. Hence, if $\xi\in\adhDR{R}$ satisfies $\dhimpsing{E}{\phi_x(\xi)}>\rho_1$, then $\dhimpsing{E}{\phi_y(\xi)}>\frac{\rho_1}{2}$. Take $\xi\in\set{D}$ such that
  \[1-\Cmod{\xi}^2=8\eps_2^{-1}\Cmod{\sin(\theta/2)}=\frac{1}{2}\eps_1^{-1}\Cmod{\sin(\theta/2)}.\]
  Since $y$ is still far from the singular set, it is clear by Lemma~\ref{lemdefeps0} that $A\subset\intcc{-C\eps}{C\eps}$, for some constant $C>1$. Shrinking~$\eps$ if necessary, we obtain that $\Cmod{\xi}\geq\frac{1}{4}$. By Lemma~\ref{distrotid},
  \[2c\eps=\eps_1\leq\dPCp{\xi}{e^{i\theta}\xi}\leq\eps_2<\eps_0.\]
  Hence,~\eqref{eqtheta12} together with Lemma~\ref{lemdefeps0} imply that either $\dPC{0}{\xi}>R$, or $\phi_y(\xi)\in\frac{1}{2}U_1^{\sing}$. This should hold for every~$\xi$ with same modulus, so we deduce by Lemma~\ref{lemRj12}\eqref{Rjsing1} that $\dPC{0}{\xi}\in\cup_{i=0}^N\intoo{R_{i,1}}{R_{i,2}}$. Therefore, $1-\Cmod{\xi}^2\in\cup_{i=0}^N\intoo{2e^{-R_{i,2}}}{4e^{-R_{i,1}}}$ and
  \[\sin(\theta/2)\in\bigcup_{i=0}^N\intoo{8c\eps e^{-R_{i,1}}}{16c\eps e^{-R_{i,2}}}.\]
  This concludes the proof.
\end{proof}

\begin{proof}[End of proof of Lemma~\ref{estimLebtheta}] For the upper bound, we apply Lemmas~\ref{ensthetapasprescrits} and~\ref{mesepspresc}. Note that $\eps_{i,2}<\frac{\eps_{i-1,1}}{2}$, because $R_{i,2}-R_{i-1,1}>4$ in Lemma~\ref{lemRj12}. We obtain
  \[\Leb(A)\leq 2\eps_{N,2}3^N\prod\limits_{i=0}^{N-1}\frac{\eps_{i,2}}{\eps_{i,1}}\leq 32c\eps e^{-R}e^{R-R_{N,2}}12^N\prod\limits_{i=0}^{N-1}e^{R_{i,1}-R_{i,2}}.\]
  The last statement of Lemma~\ref{lemRj12} now gives
  \[\Leb(A)\leq 32c\eps e^{\delta(3\log(12)+12)R}e^{-R}.\]
  Since~$\delta$ was chosen arbitrarily, we deduce the upper bound wanted. For the lower bound, we will use notations of Lemma~\ref{distrotid}. Recall also that $c_0>0$ is such that $\eta\leq c_0$ on $\manis{M}{E}$. If $\theta$ satisfies $\Cmod{\sin(\theta/2)}\leq c_0^{-1}\frac{\eps}{4}e^{-R}$ and $\xi\in\adhDR{R}$, then $\frac{1-\Cmod{\xi}^2}{2\Cmod{\xi}\Cmod{\sin(\theta/2)}}\geq\frac{4c_0}{\eps}$. Thus,
  \[\dPCp{\xi}{e^{i\theta}\xi}=f\left(\frac{1-\Cmod{\xi}^2}{2\Cmod{\xi}\Cmod{\sin(\theta/2)}}\right)\leq f\left(\frac{4}{c_0\eps}\right)\leq c_0^{-1}\eps,\]
  where~$f$ is defined in Lemma~\ref{distrotid}. Since~$\xi$ is arbitrary in~$\adhDR{R}$, we get that $\theta\in A$. It follows that $\Leb(A)\geq c_0^{-1}\frac{\eps}{2}e^{-R}$. Such an estimate can also be thought as Lemma~\ref{distrotid} for $\eps_2=c_0^{-1}\eps$ and $\eps_1=0$.
\end{proof}

We argue similarly for Lemma~\ref{estimLebzeta}, with analogous steps. But first, let us see how to compose the automorphisms of the disk.

\begin{lem}\label{compzeta12theta12} Let $\zeta_1,\zeta_2\in\Aut(\set{D})$ and $\theta_1,\theta_2\in\set{R}$. If $\zeta_1,\zeta_2$ are sufficiently small, there exist $\zeta\in\set{D}$ and $\theta\in\set{R}$, with
  \[\frac{1}{2}\norm{(\zeta_1-\zeta_2,\theta_1-\theta_2)}_{\infty}\leq\norm{(\zeta,\theta)}_{\infty}\leq3\norm{(\zeta_1-\zeta_2,\theta_1-\theta_2)}_{\infty},\]
  satisfying
  \[\left(\tau_{\zeta_1}\circ r_{\theta_1}\right)^{-1}\circ\left(\tau_{\zeta_2}\circ r_{\theta_2}\right)=\tau_{\zeta}\circ r_{\theta}.\]
  Here, $\norm{(\zeta,\theta)}_{\infty}=\max\left(\Cmod{\zeta},\Cmod{\theta}\right)$.
\end{lem}

\begin{proof} This is just computation, that we partly leave to the reader. We have the following composition rules.
  \[r_{\theta}\circ\tau_{\zeta}=\tau_{e^{i\theta}\zeta}\circ r_{\theta},\qquad \tau_{\zeta_1}\circ\tau_{\zeta_2}=r_{\theta(\zeta_1,\zeta_2)}\circ\tau_{\zeta(\zeta_1,\zeta_2)},\]
  for $\zeta(\zeta_1,\zeta_2)=\tau_{\zeta_2}(\zeta_1)$ and $\theta(\zeta_1,\zeta_2)$ satisfying $e^{i\theta(\zeta_1,\zeta_2)}=\frac{1+\zeta_1\cjg{\zeta_2}}{1+\cjg{\zeta_1}\zeta_2}$. Playing with these rules, we get the form wanted, with $\Cmod{\zeta}=\Cmod{\tau_{\zeta_2}(-\zeta_1)}$. So indeed, $\dPC{0}{\zeta}=\dPC{\zeta_1}{\zeta_2}$ and $2\Cmod{\zeta_1-\zeta_2}\leq\Cmod{\zeta}\leq3\Cmod{\zeta_1-\zeta_2}$. Moreover, $\theta=\theta_2-\theta_1+\theta(-\zeta_1,\zeta_2)$. Therefore, we just need to show that $\Cmod{\theta(-\zeta_1,\zeta_2)}\leq\Cmod{\zeta_1-\zeta_2}$. We have
  \[\Cmod{1-e^{i\theta(-\zeta_1,\zeta_2)}}=\Cmod{\frac{2i\Im\left(\zeta_1\cjg{\zeta_2}\right)}{1+\cjg{\zeta_1}\zeta_2}}=\Cmod{\frac{2\Im\left((\zeta_1-\zeta_2)\cjg{\zeta_2}\right)}{1+\cjg{\zeta_1}\zeta_2}}\leq\frac{1}{2}\Cmod{\zeta_1-\zeta_2},\]
      if $\zeta_1$ and $\zeta_2$ are sufficiently small. Above, $\Im(z)$ denotes the imaginary part of a complex number~$z$. We get the estimate wanted.
    \end{proof}

    \begin{lem}\label{disttauzeta} There exist constants $C_1,C_2>1$ satisfying the following. Let $0<\eps_1<\eps_2$ be sufficiently small with $\eps_1\leq(C_1C_2)^{-1}\eps_2$, $\theta\in\intcc{-\pi}{\pi}$ and $\zeta\in\rD{r_0}$. There exists a subset $\Lambda\subset\intcc{-\pi}{\pi}$, with $\Leb(\Lambda)=\frac{\pi}{2}$, such that if $\xi=\rho e^{i\lambda}$, $\rho\geq\frac{1}{4}$, $\lambda\in\Lambda$ and
      \[C_2\eps_2^{-1}\norm{(\zeta,\theta)}_{\infty}\leq1-\Cmod{\xi}^2\leq C_1^{-1}\eps_1^{-1}\norm{(\zeta,\theta)}_{\infty},\]
      then $\dPCp{\xi}{\tau_{\zeta}\left(e^{i\theta}\xi\right)}\in\intcc{\eps_1}{\eps_2}$.
    \end{lem}

    \begin{proof} First, let us compute a formula for the distance
      \[\begin{aligned}\dPCp{\xi}{\tau_{\zeta}\left(e^{i\theta}\xi\right)}&=2\tanh^{-1}\Cmod{\left(\xi-\frac{\zeta+e^{i\theta}\xi}{1+\cjg{\zeta}e^{i\theta}\xi}\right)\left(1-\cjg{\xi}\frac{\zeta+e^{i\theta}\xi}{1+\cjg{\zeta}e^{i\theta}\xi}\right)^{-1}}\\
          &=2\tanh^{-1}\Cmod{\frac{\xi(1-e^{i\theta})-\zeta+\cjg{\zeta}e^{i\theta}\xi^2}{1-e^{i\theta}+e^{i\theta}\left(1-\Cmod{\xi}^2\right)+\cjg{\zeta}e^{i\theta}\xi-\zeta\cjg{\xi}}}\\
          &=2\tanh^{-1}\Cmod{\frac{2i\xi\sin(\theta/2)+\zeta e^{-i\theta/2}-\cjg{\zeta}e^{i\theta/2}\xi^2}{2i\left(\sin(\theta/2)+\Im\left(\zeta\cjg{\xi}e^{-i\theta/2}\right)\right)-e^{i\theta/2}\left(1-\Cmod{\xi}^2\right)}}.\end{aligned}\]
      Now, take $\xi=\sigma i\frac{\zeta}{\Cmod{\zeta}}e^{-i\theta/2}\rho e^{i\lambda}$, with $\sigma=\pm1$ chosen such that $\sigma\sin(\theta/2)=-\Cmod{\sin(\theta/2)}$. If $\zeta=0$, then any complex number of modulus~$1$ can replace $\frac{\zeta}{\Cmod{\zeta}}$. The above formula becomes
      \[\dPCp{\xi}{\tau_{\zeta}\left(e^{i\theta}\xi\right)}=2\tanh^{-1}\Cmod{\frac{2\rho e^{i\lambda}\Cmod{\sin(\theta/2)}+\Cmod{\zeta}\left(1+\rho^2e^{2i\lambda}\right)}{2i\left(\sin(\theta/2)-\sigma\rho\Cmod{\zeta}\cos(\lambda)\right)-e^{i\theta/2}\left(1-\rho^2\right)}}.\]
      Next, suppose that $\Cmod{\lambda}\leq\frac{\pi}{4}$. This gives indeed a subset $\Lambda$ of angles, of Lebesgue measure~$\frac{\pi}{2}$. It should be noted that the three terms in the numerator have non-negative real part. Moreover, if $\rho$ satisfies the conditions of the lemma with $\eps_1$ and $\eps_2$ sufficiently small, it is clear that the term $1-\rho^2$ in the denominator dominates all the others. Therefore,
      \[\begin{aligned}2\tanh^{-1}\left(\frac{\sqrt{2}/(4\pi)\norm{(\zeta,\theta)}_{\infty}}{2C_1^{-1}\eps_1^{-1}\norm{(\zeta,\theta)}_{\infty}}\right)&\leq\dPCp{\xi}{\tau_{\zeta}\left(e^{i\theta}\xi\right)}\leq2\tanh^{-1}\left(\frac{3\norm{(\zeta,\theta)}_{\infty}}{C_2\eps_2^{-1}/2\norm{(\zeta,\theta)}_{\infty}}\right)\\
          \eps_1=\frac{\sqrt{2}}{8\pi}C_1\eps_1&\leq\dPCp{\xi}{\tau_{\zeta}\left(e^{i\theta}\xi\right)}\leq24C_2^{-1}\eps_2=\eps_2,\end{aligned}\]
      for $\eps_1,\eps_2$ sufficiently small. Here we have defined the constants $C_1,C_2$ to get the extreme identities. The lemma is proven.
    \end{proof}

    \begin{lem}\label{zetapasprescrits}Take the notations of Lemmas~\ref{lemdefeps0} and~\ref{lemRj12} and let $\eps>0$ be sufficiently small, $R>0$ be sufficiently large. Define $\eps_{i,1}=\frac{4}{3}cC_1\eps e^{-R_{i,1}}$ and $\eps_{i,2}=16cC_1\eps e^{-R_{i,2}}$, for $i\in\intent{0}{N}$. Note that $\eps_{N,1}=0$. Then, for $y\in B_R(x,\eps_0)$,
  \[A=\left\{(\zeta,\theta)\in\rD{r_0}\times\intcc{-\pi}{\pi};~\dhimps{\adhDR{R}}{\phi_y\circ\tau_{\zeta}\circ r_{\theta}}{\phi_y}<\eps\right\}\]
  has prescribed $(\eps_{i,j})_{i,j}$-steps.
\end{lem}

\begin{proof} We argue similarly to Lemma~\ref{ensthetapasprescrits}. Fix $\eps_2<\eps_0$ sufficiently small, $\eps_1=(C_1C_2)^{-1}\eps_2$ and $\eps=(2c)^{-1}\eps_1$. Take $(\zeta_1,\theta_1),(\zeta_2,\theta_2)\in A$ and $(\zeta,\theta)$ given by Lemma~\ref{compzeta12theta12}. That way,
  \[2\eps>\dPCs{\adhDR{R}}{\phi_y\circ\tau_{\zeta_1}\circ r_{\theta_1}}{\phi_y\circ\tau_{\zeta_2}\circ r_{\theta_2}}\geq\dPCs{\adhDR{R-3\Cmod{\zeta_1}}}{\phi_y\circ\tau_{\zeta}\circ r_{\theta}}{\phi_y}.\]
  Finally, take $\xi=\rho e^{i\lambda}$, with $\lambda\in\Lambda$ given by Lemma~\ref{disttauzeta} and
  \[1-\rho^2=C_1^{-1}\eps_1^{-1}\norm{(\zeta,\theta)}_{\infty}=C_2\eps_2^{-1}\norm{(\zeta,\theta)}_{\infty}.\]
  As in Lemma~\ref{ensthetapasprescrits}, because $\Leb(\Lambda)=\frac{\pi}{2}$, we obtain $1-\rho^2\in\cup_{i=0}^N\intoo{2e^{-R_i,1}}{4e^{-R_{i,2}}}$. Here, we work with Lemma~\ref{lemRj12} for $R-3\Cmod{\zeta_1}$ but this does not change anything to the estimates. This gives $\norm{(\zeta,\theta)}_{\infty}\in\intoo{4cC_1\eps e^{-R_{i,1}}}{8c\eps C_1 e^{-R_{i,2}}}$. Lemma~\ref{compzeta12theta12} then implies
    \[\norm{(\zeta_1-\zeta_2,\theta_1-\theta_2)}_{\infty}\in\bigcup\limits_{i=0}^N\intoo{\frac{4}{3}cC_1\eps e^{-R_{i,1}}}{16c C_1\eps e^{-R_{i,2}}}.\qedhere\]
  \end{proof}

  \begin{proof}[End of proof of Lemma~\ref{estimLebzeta}] The proof is essentially the same as of Lemma~\ref{estimLebtheta}. Note that $\eps_{i,2}<\frac{\eps_{i-1,1}}{2}$ because $R_{i,2}-R_{i-1,1}>4$. By Lemmas~\ref{mesepspresc} and~\ref{zetapasprescrits}, there exist universal constants $C,C'$ such that the Lebesgue measure of the Borel set involved is lower than $C\eps^3e^{(C'\delta-3)R}$. For the lower bound, the computation of Lemma~\ref{disttauzeta} easily shows that for $\norm{(\zeta,\theta)}_{\infty}\leq C\eps e^{-R}$ and $\xi\in\adhDR{R}$, $\dPCp{\xi}{\tau_{\zeta}\left(e^{i\theta}\xi\right)}\leq c_0^{-1}\eps$. We conclude the same using that~$\fol$ is Brody-hyperbolic.
  \end{proof}

\end{document}